\documentclass[12pt]{article}

\usepackage[utf8]{inputenc}
\usepackage{cite}
\usepackage{amssymb}
\usepackage{amsthm}
\usepackage{amsmath}
\usepackage{amsfonts}
\usepackage{enumerate}
\usepackage{hyperref}
\usepackage{todonotes}

\usepackage{tikz}
\usepackage{authblk}

\hypersetup{
  pdftitle = {Well-quasi-ordering H-contraction-free graphs},
  pdfauthor = {M.\ Kaminski, J.-F.\ Raymond and T.\ Trunck},
  colorlinks = true,
  linkcolor = black!30!red,
  citecolor = black!30!green
}

\usetikzlibrary{decorations.pathreplacing}
\tikzset{black node/.style={draw, circle, fill = black, minimum size = 5pt, inner sep = 0pt}}
\tikzset{white node/.style={draw, circle, fill = white, minimum size = 5pt, inner sep = 0pt}}
\tikzset{normal/.style = {draw=none, fill = none}}

\newtheorem{theorem}{Theorem}
\newtheorem{lemma}{Lemma}
\newtheorem{corollary}{Corollary}
\newtheorem{proposition}{Proposition}

\theoremstyle{remark}

\newtheorem{observation}{Observation}
\newtheorem{claim}{Claim}

\theoremstyle{definition}

\newcommand{\N}{\mathbb{N}}
\newcommand{\intv}[2]{\left \{ #1,\dots, #2 \right \}}
\DeclareMathOperator{\powset}{\mathcal{P}}

\newcommand{\gem}{\mathrm{gem}}

\newcommand{\lctr}{\mathbin{\leq_\mathrm{ctr}}} % contraction relation
\newcommand{\lleq}{\preceq} % unspecified relation
\DeclareMathOperator{\excl}{Excl}
\DeclareMathOperator{\incl}{Incl}

\newcommand{\seqb}[1]{\left \langle #1 \right \rangle} % sequence brackets
\DeclareMathOperator{\rt}{root} % root of a graph
\DeclareMathOperator{\cycle}{\mathsf{cycle}}
\DeclareMathOperator{\stick}{\mathsf{stick}}
\DeclareMathOperator{\clique}{\mathsf{clique}}
\DeclareMathOperator{\dec}{\mathsf{dec}}

\title{Well-quasi-ordering $H$-contraction-free graphs\thanks{The
    research was partially supported by the Foundation for Polish
    Science (Marcin Kamiński), the (Polish)
    National Science Centre grants SONATA UMO-2012/07/D/ST6/02432
    (Marcin Kamiński and Jean-Florent Raymond) PRELUDIUM
    2013/11/N/ST6/02706 (Jean-Florent Raymond), and by the Warsaw
    Center of Mathematics and Computer Science (Jean-Florent Raymond
    and Théophile Trunck).
    Emails: \href{mailto:mjk@mimuw.edu.pl}{\texttt{mjk@mimuw.edu.pl}},
    \href{mailto:jean-florent.raymond@mimuw.edu.pl}{\texttt{jean-florent.raymond@mimuw.edu.pl}},
    and \href{mailto:theophile.trunck@ens-lyon.org}{\texttt{theophile.trunck@ens-lyon.org}}.
  }}
\author[a]{Marcin Kamiński}
\affil[a]{Institute of Computer Science, University of Warsaw, Poland.}
\author[a,b]{Jean-Florent Raymond}
\affil[b]{LIRMM, University of Montpellier, France.}
\author[c]{Théophile Trunck}
\affil[c]{LIP, ÉNS de Lyon, France.}
\date{}%\today}

\begin{document}
\maketitle

\abstract{
A \emph{well-quasi-order} is an order which contains no infinite
decreasing sequence and no infinite collection of incomparable
elements. In this paper, we consider graph classes defined by
excluding one graph as contraction. More precisely, we give a complete
characterization of
graphs $H$ such that the class of $H$-contraction-free graphs is
well-quasi-ordered by the contraction relation.
This result is the contraction analogue of the previous dichotomy theorems of
Damsaschke [\emph{Induced subgraphs and well-quasi-ordering},
Journal of Graph Theory, 14(4):427–435, 1990] on the induced subgraph
relation, Ding [\emph{Subgraphs
  and well-quasi-ordering}, Journal of Graph Theory, 16(5):489–502,
1992] on the subgraph relation, and Błasiok et al.\ [\emph{Induced
  minors and well-quasi-ordering}, ArXiv e-prints, 1510.07135, 2015]
on the induced minor relation.
}

\section{Introduction}
\label{sec:intro}
A \emph{well-quasi-order} is a quasi-order where every decreasing
sequence and every collection of incomparable elements (called
an \emph{antichain}) are finite.
Well-quasi-orders enjoy nice combinatorial properties that can be used
in several contexts, from algebra to computational complexity and
algorithms. Since its introduction more than sixty years ago, the
theory of well-quasi-orders led to major results in Graph
Theory and Combinatorics.  In particular, Kruskal showed in \cite{Kruskal60well} that trees are well-quasi-ordered by homeomorphic embedding, and Robertson and Seymour proved that both the minor relation and the immersion relation are well-quasi-orders on the class of finite graphs \cite{Robertson2004325, Robertson2010181}. Most of the usual quasi-orders on graphs are not well-quasi-orders in general,
though. Given one of these quasi-orders, a natural line of research is to identify
the subclasses that are well-quasi-ordered.
Our work is motivated by the following results.

\begin{theorem}[\!\!\cite{JGT:JGT3190140406}]
  The class of $H$-induced subgraph-free graphs is well-quasi-ordered
  by induced subgraphs
  iff $H$ is an induced subgraph of~$P_4$.\footnote{$P_n$ is the path
    on $n$ vertices, for every $n\in \N$.\label{ft}}
\end{theorem}

\begin{theorem}[\!\!\cite{Ding:1992:SW:152782.152791}]
  The class of $H$-subgraph-free graphs is well-quasi-ordered by subgraphs
  iff $H$ is a subgraph of~$P_n$, for some $n\in \N$.\textsuperscript{\rm \ref{ft}}
\end{theorem}

\begin{theorem}[\!\!\cite{Lui2014}]
  The class of $H$-topological minor-free multigraphs is
  well-quasi-ordered by topological minors iff $H$ is a topological
  minor of $R_n$, for some~$n \in \N$.\footnote{$R_n$ is the multigraph obtained by doubling
    every edge of a path on $n$ edges, for every $n\in \N$.}
\end{theorem}

\begin{theorem}[\!\!\cite{2015arxivim}]
  The class of $H$-induced minor-free graphs is
  well-quasi-ordered by induced minors iff $H$ is an induced minor of
  the $\gem$ or~$\hat{K_4}$.\footnote{The $\gem$ is the graph obtained by adding a
dominating vertex to $P_4$ and $\hat{K_4}$ is the graph obtained
by adding a vertex of degree 2 to~$K_4$.}
\end{theorem}

These results characterize the closed classes defined by one forbidden
substructure that are well-quasi-orders. Like the four containment relations on
graphs mentioned in the above results, the contraction relation is not a
well-quasi-order in general. Let the diamond be the graph obtained from $K_4$ by deleting
an edge. Our main contribution in this direction is the following
result.

\begin{theorem}\label{main}
  The class of connected $H$-contraction-free graphs is well-quasi-ordered by
  contractions iff $H$ is a contraction of the diamond.
\end{theorem}

The requirement of connectivity in \autoref{main} is necessary in the
sense that for every graph $H$, the class of (not necessarily
connected) $H$-contraction-free graphs contains the infinite antichain
$\{\overline{K}_{i},\ i \in \N_{\geq h+1}\}$ (where $h = |V(H)|$) and
therefore is not a well-quasi-order. \autoref{main} can be seen as
contraction counterpart of the results mentioned above.

Another line of research when dealing with quasi-orders that are not
well-quasi-orders in general is to look at canonical antichains. An antichain is
\emph{canonical} if for every closed subset $F$ of the quasi-order, $F$ is a
wqo iff $F$ has a finite intersection with this
antichain. Intuitively, a canonical antichain represents all infinite
antichains of a quasi-order.
As shown by the results below, the question of the presence or absence
of a canonical antichain has been studied for several containment
relations and graph classes.

\begin{theorem}[\!\!\cite{Ding20091123}]
  Under the subgraph relation, the class of finite graphs has a
  canonical antichain.
\end{theorem}

\begin{theorem}[\!\!\cite{Ding20091123}]
  Under the induced subgraph relation, the class of finite graphs does
  not have a canonical antichain.
\end{theorem}

\begin{theorem}[\!\!\cite{Ding20091123}]
  Under the induced subgraph relation, both the class of interval
  graphs and the class of bipartite permutation graphs have a
  canonical antichain.
\end{theorem}

\begin{theorem}[\!\!\cite{2014arXiv1412.2407K}]
  Under the multigraph contraction relation, the class of finite (loopless)
  multigraphs has a canonical antichain.
\end{theorem}

We give an answer to this question for the containment relation with
the following result.

\begin{theorem}\label{cactr}
Under the contraction relation, the class of finite graphs does not
have a canonical antichain. 
\end{theorem}

The proof of \autoref{cactr} relies on the tools introduced in
\cite{Ding20091123} that can be used to prove that a quasi-order does
not have a canonical antichain.

\paragraph{Organization of the paper.}
The proof of \autoref{main} contains three parts. The first one, given in
\autoref{sec:ia}, is a study of infinite antichains of the contraction
relation from which we can 
deduce that if the class of $H$-contraction-free graphs is
well-quasi-ordered by contractions, then $H$ is a contraction of the
diamond. \autoref{sec:ddcfg} contains the second part which is a
decomposition theorem for diamond-contraction-free graphs. The last part uses this decomposition to show the well-quasi-ordering
result and is presented in \autoref{sec:wqoccg}.
The proof of~\autoref{cactr} is given in \autoref{sec:ca}.
Definitions of the terms and notations
used are introduced in~\autoref{prelim}.

\section{Preliminaries}\label{prelim}

We use the notation $\N_{\geq k}$ for the set $\{i\in \N,\ i \geq k\}$,
for every~$k\in \N$. For every set $S$, we denote by $\powset(S)$ the
collection of subsets of~$S$.

\paragraph{Graphs.}
All graphs in this paper are finite, simple, and undirected. We denote by $V(G)$ the vertex set of a graph $G$ and by $E(G)$ its
edge set. If $X \subseteq V(G)$, the \emph{subgraph of
  $G$ induced by $X$}, which we write $G[X]$, is the graph with vertex
set $X$ and edge set~$E(G) \cap X^2$.
Let $C$ be a (not necessarily induced) cycle in a graph $G$. A pair of
vertices $\{u,v\} \subseteq V(C)$ that are not adjacent in $C$ is
a \emph{chord of $C$ in $G$} if $\{u,v\} \in E(G)$. Otherwise
$\{u,v\}$ is a \emph{non-chord of $C$ in~$G$}.

A vertex $v$ of a graph $G$ is a \emph{cutvertex} if $G \setminus
\{v\}$ has more connected components than~$G$.
A \emph{block} is a maximal subgraph that has no cutvertex. A
\emph{clique-cactus graph} is a graph whose blocks are cycles and
cliques (cf.\ \autoref{cactus} for an example).

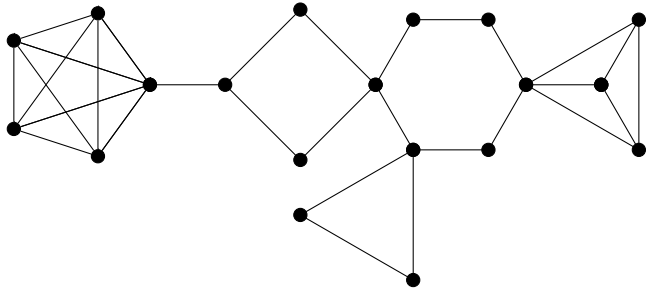
\begin{figure}[h]
  \centering
      \begin{tikzpicture}[every node/.style = black node]
      \draw (0:1) node {} --  (60:1) node {} --  (120:1) node {} --  (180:1) node {} --  (240:1) node {} --  (300:1) node {} -- cycle;
      \draw[xshift = 2cm] (60:1) node {} -- (180:1) node {} -- (-60:1) node {} -- cycle (0,0) node {} -- (60:1) (0,0) node {} -- (-60:1) (0,0) node {} -- (180:1);
      \draw[shift = (240:2)] (60:1) node {} -- (180:1) node {} -- (-60:1) node {} -- cycle;
      \draw[xshift = -2cm] (0:1) node {} -- (90:1) node {} -- (180:1) node {} -- (-90:1) node {} -- cycle;
      \begin{scope}[xshift = -5cm]
        \foreach \i in {0,..., 5}{
          \draw (\i*72:1) node (N\i) {};
        }
        \foreach \i in {0, ..., 5}{
        \foreach \j in {\i, ..., 5}{
          \draw (N\i) -- (N\j);
        }         
        }
        \draw (0:1) -- +(1,0);
      \end{scope}
      
    \end{tikzpicture}
  \caption{A clique-cactus graph.}
  \label{cactus}
\end{figure}

If $G$ is a graph, then $\overline{G}$ is the graph obtained by
replacing all non-edges by edges and vice versa.
For every positive integer $r$ we denote by $D_r$ the graph $\overline{2\cdot
K_1 \cup \cdot K_r}$. In particular, $D_2$ is the diamond. We set $\mathcal{D} = \{D_r,\ r \in \N\}$ (cf.\ \autoref{d2}) and $\mathcal{S} = \{K_{1,r},\  r \in \N\}$.

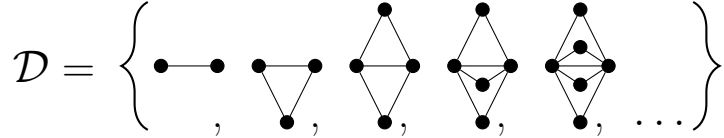
\begin{figure}[h]
  \centering
    \begin{tikzpicture}[every node/.style = black node]
      \begin{scope}[yshift = -2cm]
    \draw[xshift = -1.3cm, scale = 0.5] (-0.75, -0.5) node (a1) {} -- (0.75, -0.5) node (a2) {} (0.75, -2.2) node[normal] {{\Large,}};
    \draw[xshift = 0cm, scale = 0.5] (-0.75, -0.5) node (a1) {} -- (0.75, -0.5) node (a2) {} -- (0,-2) node {} -- (a1)
    (0.75, -2.2) node[normal] {{\Large,}};
    \draw[xshift = 1.3cm, scale = 0.5] (-0.75, -0.5) node (a1) {} --
    (0.75, -0.5) node (a2) {} -- (0,1) node {} -- (a1) -- (0,-2) node {} -- (a2)
    (0.5, -2.2) node[normal] {{\Large,}};
    \draw[xshift = 2.6cm, scale = 0.5] (-0.75, -0.5) node (a1) {} --
    (0.75, -0.5) node (a2) {} -- (0,1) node {} -- (a1) -- (0,-2) node
    {} -- (a2) -- (0,-1) node {} -- (a1)
    (0.5, -2.2) node[normal] {{\Large,}};
    \draw[xshift = 3.9cm, scale = 0.5] (-0.75, -0.5) node (a1) {} --
    (0.75, -0.5) node (a2) {} -- (0,1) node {} -- (a1) -- (0,-2) node
    {} -- (a2) -- (0,-1) node {} -- (a1)-- (0,0) node {} -- (a2)
    (0.5, -2.2) node[normal] {{\Large,}};

    \begin{scope}[xshift = -2.6cm]
      \node[normal] at (-0.5, -0.25) {{\Large$\mathcal{D} = $}};
    \node[normal] at (7.6, -1) {{\Large$\dots$}};
    \draw[thick,decorate,decoration={brace,amplitude=10pt}] (0.75,-1.1)
    -- (0.75,0.6);
    \draw[thick,decorate,decoration={brace,amplitude=10pt}] (8, 0.6) --
    (8, -1.1);
    \end{scope}
  \end{scope}
  \end{tikzpicture}
  \caption{Graphs of $\mathcal{D}$.}
  \label{d2}
\end{figure}

\paragraph{Subsets of vertices.} If $G$ is a graph, the \emph{degree}
of a subset $S \subseteq V(G)$ is the number of vertices of $V(G)
\setminus S$ that have a neighbor in $S$. The subset $S$ is said to be
\emph{connected} if $G[S]$ is connected. We say that $S$ is
\emph{adjacent} to some vertex $v$ (respectively some subset $S'
\subseteq V(G)$) if there is an edge from $v$ to a vertex of $S$
(respectively from a vertex of $S$ to a vertex of $S'$).

\paragraph{Contractions.}
In a graph $G$, a \emph{contraction of the edge $\{u,v\} \in E(G)$} is the
operation which adds a new vertex adjacent to
the neighbors of $u$ and $v$ and then deletes $u$ and~$v$. We say that
a graph $H$ is a \emph{contraction} of a graph $G$ whenever $H$ can be
obtained from $G$ by a sequence of edge contractions, what we denote by~$H \lctr G$.

A \emph{contraction model} of a graph $H$ in a graph $G$ is function
$\varphi \colon V(H) \to \powset(V(G))$ such that:
\begin{enumerate}[(i)]
\item for every $v\in V(H)$, $\varphi(v)$ is connected;
\item $\{\varphi(v),\ v \in V(H)\}$ is a partition of $V(G)$;
\item for every $u,v \in V(H)$, the vertices $u$ and $v$ are adjacent
  in $H$ iff the subsets $\varphi(u)$ and $\varphi(v)$ are adjacent in~$G$.
\end{enumerate}
This definition has several consequences. In particular, the degree of
$v\in V(H)$ in $H$ is at most the degree of $\varphi(v)$ in~$G$.
Also, there is no model of a graph with no dominating vertex in a
graph with a dominating vertex.

It is easy to check that $H$ is a contraction of $G$ iff there is a
contraction model of $H$ in~$G$.
A graph $G$ is said to \emph{exclude a graph $H$ as contraction}, or to be
\emph{$H$-contraction-free}, if $H$ is not a contraction of $G$. We
denote the class of connected $H$-contraction-free graphs by~$\excl(H)$.

We say that a graph $H$ is an \emph{induced minor} of a graph $G$ if
it can be obtained from $G$ by deleting vertices and contracting edges.

\paragraph{Sequences and orders.}
We write $\seqb{s_1, \dots, s_n}$ the sequence containing the
elements $s_1, \dots, s_n$ in this order. For every set $S$, we denote by 
$S^\star$ the set of all finite sequences over $S$, including the
empty sequence.
For any partial order $(A, \lleq),$ we define
the relation $\lleq^\star$ on $A^\star$ as follows: for every $r
=\seqb{r_1,\dots, r_p}$ and $s = \seqb{s_1,\dots, s_q}$ of
$A^\star,$ we have $r \lleq^\star s$ if there is an increasing
function $\varphi \colon \intv{1}{p} \to \intv{1}{q}$ such that for
every $i \in \intv{1}{p}$ we have~$r_i \lleq s_{\varphi(i)}.$ This
generalizes the subsequence relation.

\paragraph{Well-quasi-orders and antichains.}
Given an order $\lleq$ over $S$, a sequence over $S$
is said to be an antichain of $(S, \lleq)$ if its elements
are pairwise incomparable with respect to~$\lleq$.
A \emph{well-quasi-order} (wqo for short) is a quasi-order where every decreasing
sequence and every antichain is finite.

We will use the two classical results stated below.
\begin{proposition}[Folklore]\label{union}
  Let $(A, \lleq)$ be a quasi-order and $B,C \subseteq A$. If both
  $(B, \lleq)$ and $(C, \lleq)$ are wqo, then so is $(B
  \cup C, \lleq)$.
\end{proposition}

\begin{proposition}[Higman's Lemma \cite{Higman01011952}]\label{higman}
  $(A, \lleq_A)$ is a wqo, then so is $(A^\star, \lleq^\star_A)$.
\end{proposition}

If $(S, \lleq)$ is a quasi-order that is not a wqo, a \emph{minimal
  antichain} \cite{Nash63onwe} of $(S, \lleq)$ is an antichain $\seqb{a_i}_{i \in \N}$ where for
every $i \in \N$, $a_i$ is a minimal element (with respect to $\lleq$)
such that there is an infinite antichain of $(S, \lleq)$ starting with $\seqb{a_j}_{j \in
\intv{0}{i}}$. Observe that every quasi-order that is not a wqo and
that has no infinite decreasing sequence has a minimal antichain.
For every subset $A \subseteq S$, we define:
\[\incl(A) = \{x \in S,\ \exists y \in A,\ x \lleq y\ \text{and}\ x \neq y\}.\]
An antichain ${A}$
is \emph{fundamental} if $(\incl(A), \lleq)$ is a~wqo. A set $F
\subseteq S$ is said to be $\lleq$-closed if it satisfies the following
property: $\forall x \in F, \forall y
\in S,\ y \lleq x \Rightarrow y \in F$.

An antichain $A$ of a quasi-order $(S, \lleq)$ is \emph{canonical} if
for every $\lleq$-closed subset $F\subseteq S$, we have\[
  F \cap A\ \text{is finite}\iff (F, \lleq)\ \text{is a~wqo}.
\]

Let us end this section by a simple observation.
\begin{observation}
Every sequence of graphs that is decreasing with respect to contraction is finite.
In fact, in such a sequence the number of edges is also decreasing, as
every edge contraction decreases the number of edges of a graph by at least~one.
\end{observation}

A consequence of the above observation is that infinite antichains are
the only obstacles for a class of graphs to be well-quasi-ordered by
the contraction relation. We deal with them in the next section.

\section{Infinite antichains}
\label{sec:ia}

A simple but crucial observation in the study of the well-quasi-orderability of
classes that are defined by forbidden structures (of any kind) is the
following. If none of the graphs of an infinite antichain
$\mathcal{A}$ contains some graph $H$, then excluding $H$ does not
give a well-quasi-order. Indeed, the class obtained still contains the infinite
antichain $\mathcal{A}$. Let us restate this observation in terms of contractions.

\begin{observation}\label{o:antichains}
  Let $\mathcal{A}$ be an infinite antichain of the contraction
  relation.
  If $(\excl(H), \lctr)$ is a wqo, then all but finitely many graphs of $\mathcal{A}$
  contain $H$ as contraction.
\end{observation}

For this reason, we deal here with infinite antichains of the
contraction relation.
The simplest one is certainly the class of complete bipartite graphs with one
part of size two: $\mathcal{A}_{K} = 
\{K_{2,r},\ r\in \N_{\geq 2}\}$.% (cf.~\autoref{fig:ak}).

% \begin{figure}[ht]
%   \centering
%   \begin{tikzpicture}[every node/.style = black node]
%     \def \maxi {5}
%       \begin{scope}[yshift = -2cm]
%     \node[normal] at (-0.5, -0.25) {{\Large$\mathcal{A}_{K} = $}};
%         \foreach \nb in {1,...,\maxi} {
%           \begin{scope}[xshift = 1.3*\nb cm]
%           \begin{scope}[scale = 0.5]
%             \draw (-0.75, -0.5) node (a1) {};
%             \draw (0.75, -0.5) node (a2) {};
%             \foreach \i in {0,...,\nb}{
%               \node (W\i) at (0, 1 - \i * 3 / \nb) {};
%               \draw (a2) -- (W\i) -- (a1);
%             }
%           \end{scope}
%           \node[normal] at (0.5, -1.1) {{\Large,}};
%         \end{scope}
%     }
%     \node[normal] at (7.6, -1) {{\Large$\dots$}};
%     \draw[thick,decorate,decoration={brace,amplitude=10pt}] (0.75,-1.1)
%     -- (0.75,0.6);
%     \draw[thick,decorate,decoration={brace,amplitude=10pt}] (8, 0.6) --
%     (8, -1.1);
%   \end{scope}
%   \end{tikzpicture}
%   \caption{The antichain $\mathcal{A}_K$.}
%   \label{fig:ak}
% \end{figure}

\begin{lemma}
  For every $p,q,p',q' \in \N_{\geq 2}$ such that $p\leq p'$ and $q<q'$, there is no model of $K_{p,q}$ in $K_{p',q'}$.
\end{lemma}
\begin{proof}
  Let us assume for a contradiction that there is a model $\varphi$ of
  $K_{p,q}$ in $K_{p',q'}$. As $K_{p',q'}$ has more vertices than
  $K_{p,q}$, there is a vertex $v$ of $K_{p,q}$ such that
  $|\varphi(v)| \geq 2$. Observe that every subset of at least two
  vertices of $K_{p,q}$ that induced a connected subgraph is
  dominating. Indeed, such a subset must contain at least a vertex
  from each part of the bipartition.
  It follows from the definition of a model that $K_{p',q'}$ has a
  dominating vertex, a contradiction. Therefore, there is no model of
  $K_{p,q}$ in $K_{p',q'}$.
\end{proof}

\begin{corollary}
  $\{K_{2,p},\ p\in \N_{\geq 2}\}$ is an antichain of~$\lctr$.
\end{corollary}

Recall that $\mathcal{S}$ is the class of stars and that $\mathcal{D} = \{D_r, r\in \N\}$, where $D_r$ is the graph that can be obtained by contracting an edge of $K_{2, r+1}$, for every $r \in \N$ (cf.\ \autoref{prelim}).
The following will be useful later.
\begin{observation}\label{ctr}
  For every $p \in \N_{\geq 1}$, contracting one edge in $D_p$ gives
  either $D_{p-1}$, or $K_{1,p}$, depending on which edge is contracted.
\end{observation}

As we want to identify graphs $H$ such that $(\excl(H), \lctr)$ is a
wqo, we must consider every graph $H$ such that $\mathcal{A}\cap
\excl(H)$ is finite, for every antichain $\mathcal{A}$. A first step towards this goal is the following observation.

\begin{lemma}
  Let $p \in \N_{\geq 2}$. If $H \lctr K_{2,p}$ and $H \neq K_{2,p}$, then $H\in 
  \mathcal{D} \cup \mathcal{S}$.
\end{lemma}
\begin{proof}
  Given that $H \lctr K_{2,p}$, there is a sequence of contractions
  transforming $K_{2,p}$ into $H$. If this sequence contains only one
  contraction, then it is straightforward that $H = D_{p-1}$. Therefore
  in the other cases $H$ is a contraction of $D_{p-1}$. We get the
  result from \autoref{ctr} and the observation that every contraction
  of a graph of $\mathcal{S}$ (i.e.\ the class of stars) belongs to $\mathcal{S}$.
\end{proof}

\begin{observation}\label{dpgraph}
  For every positive integers $p,q$ such that $p<q$, we have $D_p \lctr K_{2,q}$.
\end{observation}
Indeed, if $F$ is a collection of $q-p$ edges of $K_{2,q}$ that all incident with the same vertex of degree $p$, then
it is easy to check that contracting $F$ in $K_{2,q}$ yields $D_p$.
An immediate consequence of \autoref{dpgraph} is that $\mathcal{A}_K \cap
\excl(D_p)$ is finite for every positive integer~$p$.

From the fact that every graph of $\mathcal{D} \cup \mathcal{S}$ is a
contraction of $D_p$ for some positive integer $p$, \autoref{dpgraph} gives.
\begin{observation}
  If $(\excl(H), \lctr)$ is a wqo, then $H \lctr D_p$ for some $p \in
  \N_{\geq 1}$
\end{observation}

However, we will need another antichain in order to find more
properties that $H$ must satisfy.
Let us consider the antichain of antiholes, which already appeared
in~\cite{2015arxivim} in the context of induced minors:
$\mathcal{A}_{\overline{C}} = \{\overline{C}_i,\ i \in \N_{\geq
  6}\}$ (cf.~\autoref{fig:cocycles}). This connection with the induced minor relation (where edge
contractions and vertex deletions are allowed) is not surprising: as
every contraction is an induced minor, every antichain of the induced
minor relation is also an antichain of the contraction relation.

\begin{figure}[ht]
  \centering
  \begin{tikzpicture}[every node/.style = black node]
    \begin{scope}[xshift = 2cm]
    \foreach \v in {6,...,8}{
      \pgfmathparse{int(3 * (\v - 6))}\let \tmp \pgfmathresult
      \begin{scope}[xshift = \tmp cm]
        \pgfmathparse{360 / \v} \let \angle \pgfmathresult

        \foreach \a in {1,...,\v}
        {
          \node[draw] (N\a) at (\a*\angle:1) {};
        }
        
        \foreach \a in {1,..., \v}{
          \pgfmathparse{\a + \v - 2}\let\ubound\pgfmathresult
          \pgfmathparse{\a + 2}\let\lbound\pgfmathresult
          \foreach \na in {\lbound,...,\ubound}{
            \pgfmathparse{int(mod(\na - 1, \v) + 1)}\let\tmp\pgfmathresult
            \draw (N\a) -- (N\tmp);
          }
        }
        \draw (-90:1.5) node[normal] {$\overline{C}_\v$};
        \draw (1,-1) node[normal] {\Large ,};
      \end{scope}
    }
  \end{scope}
    \node[normal] at (-0.5, 0) {{\Large$\mathcal{A}_{\overline{C}} = $}};
    \draw (10,-1) node[normal] {\Large$\dots$};
    \draw[thick,decorate,decoration={brace,amplitude=10pt}] (0.75,-1)
    -- (0.75,1);
    \draw[thick,decorate,decoration={brace,amplitude=10pt}] (11, 1) --
    (11, -1);
  \end{tikzpicture}
  \caption{Antiholes antichain.}
  \label{fig:cocycles}
\end{figure}
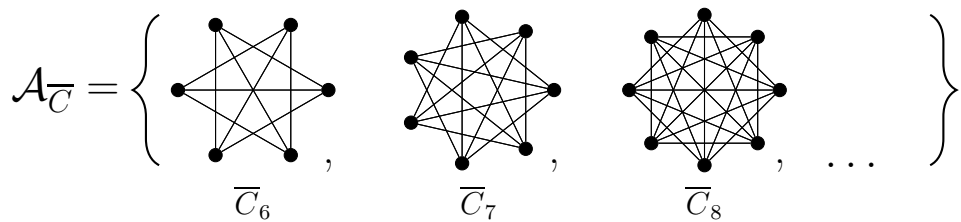

For completeness, we include the following proof.
\begin{lemma}[See also~\protect{\cite[Lemma~1]{2015arxivim}}]
   $\mathcal{A}_{\overline{C}}$ is an antichain of the contraction relation.
\end{lemma}

\begin{proof}
Towards a contradiction, let us assume that there is a contraction
model $\varphi$ of $\overline{C}_p$ in $\overline{C}_q$ for some
integers $p,q \in \N_{\geq 3}$ such that $p<q$. Recall that the image
of $\varphi$ is a partition of $\overline{C}_q$. As $|\overline{C}_p| <
|\overline{C}_q|$, there is a vertex $v$ of $\overline{C}_p$ such that
$|\varphi(v)| \geq 2$. Observe that for every choice of two
vertices of $\overline{C}_p$ there is at most one vertex which is not
adjacent to one of them. Therefore, there is at most one set in $\{\varphi(u),\ u \in
V(\overline{C}_p)\setminus \{v\}\}$ that is not adjacent to~$\varphi(v)$.
This contradicts the fact that
$\varphi$ is a model of $\overline{C}_p$ in $\overline{C}_q$ as every
vertex of $\overline{C}_q$ is adjacent to all but two
vertices. Consequently there is no contraction model of $\overline{C}_p$ in $\overline{C}_q$, for every
integers $p,q \in \N_{\geq 3}$, $p<q$.
\end{proof}

Again, we look at graphs $H$ such that $\excl(H) \cap
\mathcal{A}_{\overline{C}}$ is finite. As a wqo must contain none of
$\mathcal{A}_{K}$ and $\mathcal{A}_{\overline{C}}$, it is enough to
consider graphs such that $\excl(H) \cap
\mathcal{A}_{\overline{C}}$ is finite among those for which
$\excl(H)\cap  \mathcal{A}_{K}$ is finite.

\begin{lemma}
  If $p \geq 3$ then $\excl(D_p) \cap \mathcal{A}_{\overline{C}}$ is infinite.
\end{lemma}
\begin{proof}
  For every $p \geq 3$, then graph $D_p$ has independence number at
  least 3. Let $q>p$. As contracting edges can only decrease the independence
  number, there is no sequence of contractions transforming
  $\overline{C}_q$ (which has independent number 2) to $D_p$, for
  every integer $q > p$. Therefore $\overline{C}_q \in \excl(D_p)$,
  for every integer $q>p$.
\end{proof}

\begin{corollary}\label{direct}
  If $(\excl(H), \lctr)$ is a wqo, then $H \lctr D_2$.
\end{corollary}

The next sections are devoted to graphs not containing $D_2$ as
contraction. We will first prove a decomposition theorem for the
graphs in this class, which we will use to show that $(\excl(D_2),
\lctr)$ is a wqo.

\section{On graphs with no diamond}
\label{sec:ddcfg}

In this section we show that graphs in $\excl(D_2)$ have a simple
structure. More precisely, we prove the following lemma. Recall that clique-cactus graphs are the graphs whose blocks are cycles or cliques.

\begin{lemma}\label{dec}
  Graphs of $\excl(D_2)$ are exactly the connected clique-cactus graphs.
\end{lemma}

The proof of \autoref{dec} will be given after a few lemmas.
If $C$ is a cycle of a graph $G$ and $\{u,v\},
\{u',v'\}\subseteq V(C)$, we say that $\{u,v\}$ and $\{u',v'\}$ are
\emph{crossing in $C$} if $u,v,u',v'$ are distinct and are appearing in this order on the cycle.

\begin{lemma}\label{chord}
  Let $G$ be a graph and let $C$ be a cycle in $G$. If $C$ has at
  least one chord and one non-chord in $G$, then it has one chord and
  one non-chord that are crossing in $C$.
\end{lemma}

\begin{proof}
  Let $\{x,x'\}$ be a non-chord of $C$ in $G$ and let $P$ and $Q$ be
  the connected components of $C\setminus \{x,x'\}$ which obviously are
  paths. Let us assume that every chord of $C$ in $G$ has both
  endpoints either in $P$ or in~$Q$ (otherwise we are done) and let
  $\{y,y'\}$ be a chord of $C$ in $G$, the endpoints of which belong,
  say, to~$P$. Let $z$ be a vertex of the subpath of $P$ delimited by
  $y$ and $y'$ such that $z \not \in \{y,y'\}$, and let $z'$ be a
  vertex of~$Q$.
  If $\{z,z'\}$ is a chord of $C$ in $G$, then $\{x,x'\}$
  and $\{z,z'\}$ are satisfying the required property.
  Otherwise, $\{z,z'\}$ is a non-chord and now $\{y,y'\}$
  and $\{z,z'\}$ are crossing.
\end{proof}

\begin{lemma}\label{cycles}
  Let $G \in \excl(D_2)$. Every cycle of $G$ is either an induced
  cycle, or it induces a clique in $G$.
\end{lemma}
\begin{proof}
Let $G$ be a graph of $\excl(D_2)$ and let $C$ be a cycle of $G$.
Towards a contradiction, let us assume that $C$ has at least one chord
$\{u,u'\}$ and one non-chord $\{v,v'\}$. According to \autoref{chord} we can assume
without loss of generality that they are crossing in $C$.
Let $P$ and $Q$ be the two connected components of $C\setminus
\{v,v'\}$.
Contracting $P$ to a single vertex $x$ and $Q$ to $y$
yields a graph $G'$ such that:
\begin{itemize}
\item $v$, $x$, $v'$, $y$ lie on the cycle in this order;
\item $\{v,v'\} \not \in E(G')$; and
\item $\{x,y\} \in E(G')$ (as $\{u,u'\}$ connects the subgraphs that
  are respectively contracted to $x$ and $y$).
\end{itemize}
Notice that $G'[\{v,x,v',y\}]$ is isomorphic to $D_2$, however $G'$ may also contains other vertices.
Let us consider $G'\setminus \{v,x,v',y\}$.
While $G' \setminus \{v,v',x,y\}$ contains a connected component adjacent to $x$ or $y$, we contract it to this vertex (that we keep calling with the same name). Then, while it has a connected component adjacent to $v$ but not $v'$ (respectively $v'$ but not $v$), we contract it to $v$ (respectively $v'$), again keeping the same name for that vertex.
Finally, the only remaining connected components (if any) are adjacent to exactly $v$ and $v'$: we contract each of them to a single vertex, adjacent to $v$ and $v'$. Notice that none of these operations create an edge connecting $v$ to $v'$, thus the subset $\{v,v',x,y\}$ still induces a subgraph isomorphic to $D_2$.
Let us call $G''$ the obtained graph. As observed above, $G''$ consists of the subgraph isomorphic to $D_2$ induced by $\{v,v',x,y\}$ plus $k$ extra vertices of degree two, $z_1, \dots, z_k$, each of which is adjacent to $v$ and $v'$.
In the case where $k = 0$, $G''$ is isomorphic to $D_2$ and we reached the contradiction we were looking for.
Otherwise, we contract $\{v',y\}$ (naming the resulting vertex $v'$), which produces a complete subgraph on vertices $v,v',x$, and then we contract $\{v, z_i\}$ for every $i \in \intv{2}{k}$ (naming the resulting vertex $v$). The obtained graph is a complete graph on $v,v',x$, two vertices of which (that are $v,v'$) are adjacent to an extra vertex, $v_1$. This graph is isomorphic to $D_2$, as $x$ is not adjacent to~$v_1$, therefore we reached a contradiction.
Therefore $C$ has either
no chords or no non-chords in $G$. It is clear that in the first case
$C$ is an induced cycle of $G$ and that in the second case it induces
a clique.
\end{proof}

\begin{lemma}\label{2c}
  Let $G \in \excl(D_2)$ be a 2-connected graph. Then $G$ is either a
  cycle, or a clique.
\end{lemma}

\begin{proof}
  We assume that $|V(G)|>1$, otherwise the result is trivial.
  Let $C$ be a longest cycle of $G$. By \autoref{cycles} the cycle $C$
  is either an induced cycle, or it induces a clique in $G$. Let us
  treat these two cases separately. For contradiction we assume that
  $V(G) \setminus V(C)$ is not empty and we call $H_1, \dots, H_t$ the
  connected components of $G \setminus C$, for some $t \in \N_{\geq
    1}$. Let us consider the graph $G'$ where $H_i$, which is
  connected, has been contracted to a single vertex~$h_i$, for every
  $i \in \intv{1}{t}$. Observe that $G'$ is 2-connected, given that $G$
  is 2-connected. Also, $G' \in \excl(D_2)$.

\smallskip
\noindent \textit{First case:} $C$ induces a clique in $G'$. Notice
  that $C$ is then a maximal clique. Let $u=h_1$. As $C$ is maximal, there is a vertex $v \in V(C)$
  such that $\{u,v\} \not \in E(G')$. Let $x$ and $y$ be two neighbors
  of $u$ on $C$ (they exist since $G'$ is 2-connected). These vertices
  define two subpaths of $C$. Let $R$ be the longest of these paths that contains
  $v$. Observe that in this case, $R$ has at least three vertices. The
  union of $\{u,x\}$, $\{u,y\}$ and $R$ is a cycle of $G'$ that we
  call~$C'$. According to \autoref{cycles}, this cycle is either induced or
  it induces a clique. As $\{u,v\} \not \in E(G')$, $C'$ cannot induce
  a clique in $G$. On the other hand, $C$ is not an induced cycle as
  every pair of vertices of $R$ are adjacent (and $|V(R)| \geq 3$ as
  mentioned earlier). We reached the contradiction we were looking for.

\smallskip
\noindent \textit{Second case:} $C$ is an induced cycle and has at
  least 4 vertices.
  Let  $i \in \intv{1}{t}$. As $G'$ is 2-connected, $h_i$ has at
  least two neighbors on $C$: let $x$ and $y$ be two of them.
  \begin{claim}\label{cl}
    $x$ and $y$ are not adjacent.
  \end{claim}
  \begin{proof}
    Let us assume that $\{x,y\} \in E(G')$. Let $C'$ be the cycle
    obtained from $C$ by replacing the edge $\{x,y\}$ by the path
    $xh_iy$. This cycle is not induced as $x,y$ are not adjacent in
    $C'$ whereas $\{x,y\} \in E(G)$. It does not induce a clique
    either since $x$ is not adjacent with the other neighbor of $y$
    on $C$ (which is not $x$ as we assume that $C$ has at least 4
    vertices). This contradicts \autoref{cycles} and therefore proves that
    $\{x,y\} \not \in E(G)$.
  \end{proof}
Every pair of distinct vertices of the cycle $C$ defines two subpaths of $C$ meeting only
at these vertices. Let $u$ and $v$ be two vertices of $C$ such that
$h_i$ has at least one neighbor in the interior of each of the
subpaths of $C$ defined by $u$ and $v$, that we will respectively call
$P$ and $Q$.  Such vertices exist, as a consequence of~\autoref{cl}.
  
  Let us consider the contraction $H$ of $G'$ obtained by contracting
  the interior path of $P$ (respectively $Q$) to a single vertex $w_P$
  (respectively $w_Q$) and then by contracting the edge connecting
  $h_i$ to $w_P$. This edge exists by definition of $u$ and $v$.
  Then $uw_Pvw_Q$ is a cycle of $H$ where $\{w_P,w_Q\}$ is a chord
  (because we contracted to $w_P$ the vertex $h_1$ which was adjacent
  to both $w_P$ and $w_Q$) and $\{u,v\}$ is a non-chord (as they were
  non-adjacent vertices of the induced cycle $C$ and that nothing has
  been contracted to them). According to \autoref{cycles}, the graph
  $H$ contains $D_2$ as contraction. As $H$ is a contraction of $G$,
  then $D_2 \lctr G$, a contradiction.

\smallskip
In both cases we reached a contradiction, therefore $V(G) \setminus
V(C)$ is empty: $G$ is a clique or an induced cycle.
\end{proof}
We are now ready to prove \autoref{dec}.
\begin{proof}[Proof of \autoref{dec}.]
  The fact that a graph of $\excl(D_2)$ is clique-cactus is a
  straightforward corollary of \autoref{2c}.
  It is easy to see that a clique-cactus graph does not contain $D_2$ as
  contraction by noticing that $D_2$ is a contraction
  of a graph if and only if it is a contraction of one of its
  2-connected components. As $D_2$ is neither a contraction of a cycle,
  nor of a clique, we get the desired result.
\end{proof}

\section{Well-quasi-ordering clique-cactus graphs}
\label{sec:wqoccg}

We proved in the previous section that graphs of $\excl(D_2)$ are
exactly the connected clique-cactus graphs. This section contains the
last part of the proof of~\autoref{main}, which is the following
lemma. We conclude this section with the proof of~\autoref{main}.

\begin{lemma}\label{wqo}
  Connected clique-cactus graphs are well-quasi-ordered by~$\lctr$.
\end{lemma}

In this section, we deal with rooted graphs. A \emph{rooted graph} is a graph
which has a distinguished vertex, called \emph{root}. The contraction
relation is extended to the setting of rooted 
graphs by requiring that a model of a rooted graph $H$ in a rooted
graph $G$ maps the root of $H$ to a connected subgraph of $G$
containing the root of~$G$.

Let us denote by $\mathcal{C}$ the class of rooted connected
clique-cactus graphs. In this class, two isomorphic graphs with a
different root are seen as different. It is clear that proving that
$(\mathcal{C}, \lctr)$ is a wqo implies \autoref{wqo}. This is what we
will do.

\paragraph{Building blocks.}
Let us define three graph constructors $\stick \colon
\mathcal{C}^\star \to \mathcal{C}$, $\cycle \colon
\mathcal{C}^\star \to \mathcal{C}$, and $\clique \colon
\mathcal{C}^\star \to \mathcal{C}$. Given a sequence $\seqb{G_0, \dots,
G_{p-1}} \in \mathcal{C}^\star$ (for some $p\in \N$), if $U$ denote
the union of the graphs $G_1, \dots, G_{p-1}$, then we define;

\begin{itemize}
\item $\stick(G_0, \dots,
G_{p-1})$ is the graph obtained from $U$ by identifying the vertices $\rt(G_0),
\dots, \rt(G_{p - 1})$;
\item $\cycle(G_0, \dots,
G_{p-1})$ is the graph obtained from $U$ by adding the edges $\{\rt(G_i), \rt(G_{(i+1) \mod
  p})\}$ for every $i \in \intv{0}{p-1}$; and
\item $\clique(G_0, \dots, G_{p-1})$ is the graph obtained from $U$ by adding the
edges $\{\rt(G_i), \rt(G_j)\}$  for every distinct $i,j \in \intv{0}{p-1}$.
\end{itemize}

The root of $\stick(G_0, \dots,
G_{p-1})$, $\cycle(G_0, \dots,
G_{p-1})$ and $\clique(G_0, \dots,
G_{p-1})$ is the vertex that is the root of~$G_0$. These constructors
will allow us to encode graphs of $\mathcal{C}$ into sequences.

We will now decompose graphs of $\mathcal{C}$ along blocks.

For every block $B$ of a graph $G$,
let $\dec_B(G)$ denote the collection of all the graphs $H$ that can
be constructed from some connected component $C$ of $G\setminus
V(B)$ by adding a new vertex $v$ adjacent to the vertices 
of $C$ that are adjacent to a vertex of $B$ in $G$, and setting
$\rt(H) = v$.

Observe that as soon as $\rt(G)\in V(B)$, every graph of $\dec_B(G)$ is
a proper contraction of~$G$. Let $\dec(G)$
denote the union of the sets $\dec_B(G)$ for every block $B$ of $G$ containing
the root of $G$.
The following observation is a consequence of~\autoref{dec}. 
\begin{observation}\label{recons}
For every graph $G \in \mathcal{C}$ there
is a (not necessarily unique) sequence $\seqb{\mathcal{G}_0, \dots, \mathcal{G}_{p-1}} \in
\dec(G)^\star$ (for
some $p \in \N$) such that $G =\cycle(\stick(\mathcal{G}_0), \dots, \stick(\mathcal{G}_{p-1}))$ or $G =
\clique(\stick(\mathcal{G}_0), \dots, \stick(\mathcal{G}_{p-1}))$.
\end{observation}

\paragraph{From encodings to well-quasi-ordering.} The following lemma will allow us to
work on sequences in order to show that two graphs are comparable.

\begin{lemma}\label{cycleclique}%G=tau
  Let $\mathcal{G}, 
\mathcal{H} \in \mathcal{C}^\star$. If $\mathcal{H} \lctr^\star \mathcal{G}$, then
  \begin{enumerate}[(i)]
  \item $\cycle(\mathcal{H}) \lctr \cycle(\mathcal{G})$;\label{e:cycles}
  \item $\clique(\mathcal{H}) \lctr \clique(\mathcal{G})$; and\label{e:cliques}
  \item $\stick(\mathcal{H}) \lctr \stick(\mathcal{G})$.\label{e:stick}
  \end{enumerate}
\end{lemma}

\begin{proof}
Let $\mathcal{H} = \seqb{H_1, \dots, H_p}$ and $\mathcal{G} = \seqb{G_1, \dots,
  G_q}$ (for some positive integers $p,q$), and let $H =
\cycle(\mathcal{H})$ and $G = \cycle(\mathcal{G})$. For the sake of readability we
will refer to $H_i$'s (respectively $G_i$'s) either as elements of
$\mathcal{H}$ (respectively $\mathcal{G}$) or as subgraphs of $H$ (respectively $G$).

If $\mathcal{H} \lctr^\star \mathcal{G}$, then there is, by definition of
$\lctr^\star$, an increasing function $\varphi \colon
\intv{1}{p} \to \intv{1}{q}$ such that $\forall i \in \intv{1}{p},\
H_i \lctr G_\varphi(i)$. Therefore there is a sequence of
edge contractions transforming $G_{\varphi(i)}$ into $H_i$ for every~$i \in \intv{1}{p}$.
Let us perform the following operations on $G$:
\begin{enumerate}
\item for every $j \in \intv{1}{q} \setminus \{\varphi(i),\ i \in
  \intv{1}{p}\}$ we contract the subgraph $G_j$ to a single vertex
  $v_j$ and we then contract some edge incident with~$v_j$;\label{s1}
\item for every $i \in \intv{1}{p}$ we contract the subgraph $G_i$ in
  order to obtain the subgraph~$H_\varphi(i)$.\label{s2}
\end{enumerate}
Observe that after step~\ref{s1}., we obtain the graph $\cycle(\mathcal{G}^-)$,
where $\mathcal{G}^-$ can be obtained from $\mathcal{G}$ be deleting elements of
indices in $\intv{1}{q} \setminus \{\varphi(i),\ i \in
  \intv{1}{p}\}$. Intuitively, we contracted the graphs that do not
  appear in $H$ and removed their attachment point from the cycle.
Then we replace in step \ref{s2}.\ every graph of $\mathcal{G}^-$ by its
corresponding contraction of $\mathcal{H}$. Therefore the graph obtained
at the end is $\cycle(\mathcal{H})$, that is $H$, as required.

The cases (\ref{e:cliques}) and (\ref{e:stick}) are very similar: $H$ can be
obtained from $G$ by following the same operations as above.
\end{proof}

\begin{proof}[Proof of \autoref{wqo}]
  Let us assume by contradiction that $(\mathcal{C}, \lctr)$ is not a
  wqo. All decreasing sequences of this quasi-order are finite (as
  each contraction decreases the number of edges by at least one),
  therefore $(\mathcal{C}, \lctr)$ contains an infinite antichains. Let
  us consider a minimal antichain $\{A_i\}_{i \in \N}$ of $(\mathcal{C}, \lctr)$.
  Let $\mathcal{B} = \bigcup_{i \in \N}\dec({A}_i)$, and let us show that
  $(\mathcal{B}, \lctr)$ is a wqo. For contradiction, let us assume that
  it is not a wqo and let $\{B_i\}_{i \in \N}$ be a minimal antichain
  of this quasi-order.
  
  By definition of $\mathcal{B}$, for every $H \in \mathcal{B}$ there is
  an integer $i\in \N$ such that $H \lctr A_i$ (for instance, an
  integer $i$ such that $H \in \dec(A_i)$). Therefore for every $i \in \N$
  there is an integer $\pi(i)$ such that $B_i \lctr A_{\pi(i)}$.  Let
  $k\in \N$ be the integer such that $\pi(k)$ is minimum. Then the following sequence
  \[
\mathcal{A}=  A_0, \dots, A_{\pi(k) - 1}, B_k, B_{k+1}, \dots
  \]
  is an infinite antichain of $(\mathcal{C}, \lctr)$. Indeed, as both
  $\{A_i\}_{i \in \N}$ and $\{B_i\}_{i \in \N}$ are antichains, every
  pair of comparable graphs of $\mathcal{A}$ involves one graph of $\{A_i\}_{i \in
    \intv{1}{\pi(k)-1}}$ and one graph of $\{B_i\}_{i \in \N_{\geq k}}$.
  Let us assume that for some $i \in \intv{0}{\pi(k) - 1}$ and $j
  \in \N_{\geq k}$ we have $A_i \leq B_j$. Then $A_i \leq B_j \leq
  A_{\pi(i)}$, a contradiction with the fact that $\{A_i\}_{i \in \N}$
  is an antichain. The case $B_j \leq A_i$ is not possible by the choice
  of~$k$. This proves that $(\mathcal{B}, \lctr)$ is a wqo.
  According to \autoref{higman}, $(\mathcal{B}^\star, \lctr^\star)$ is also a
  wqo. Let $\mathcal{B}' = \{\stick(\mathcal{H}),\ \mathcal{H} \in
  \mathcal{B}^\star\}$. Item (\ref{e:stick}) of \autoref{cycleclique} implies that any
  antichain in $(\mathcal{B}', \lctr)$ can be translated into an antichain of the same length
  in~$(\mathcal{B}^\star, \lctr^\star)$, hence $(\mathcal{B}',
  \lctr)$ is a wqo. By the same argument (now using items (\ref{e:cycles}) and
  (\ref{e:cliques}) of \autoref{cycleclique}), we deduce that the quasi-orders

  \[(\{\cycle(\mathcal{H}),\ \mathcal{H} \in \mathcal{B}'^\star\},
  \lctr)\quad \text{and}\quad (\{\clique(\mathcal{H}),\ \mathcal{H} \in \mathcal{B}'^\star\},
  \lctr)\] are well-quasi-orders.
 Therefore $\mathcal{U} =
  \{\cycle(\mathcal{H}),\ \mathcal{H} \in \mathcal{B}'^\star\} \cup
  \{\clique(\mathcal{H}),\ \mathcal{H} \in \mathcal{B}'^\star\}$ is
  well-quasi-ordered by $\lctr$, as a consequence of \autoref{union}.
  According to \autoref{recons}, we have $\{A_i\}_{i \in \N} \subseteq
  \mathcal{U}$. This contradicts the fact that $\{A_i\}_{i \in \N}$ is
  an infinite antichain. Therefore $(\mathcal{C}, \lctr)$ is a wqo and
  we are done.
\end{proof}

We would like to point out that with a proof similar to \autoref{wqo}, it is in fact possible prove that if a class $\mathcal{G}$ of 2-connected graphs is wqo by $\lctr$, then so is the class of graphs whose blocks belong to $\mathcal{G}$. The interested reader may have a look at \cite[Lemma~5]{2014arXiv1412.2407K} for a result of this flavour, see also \cite[Theorem~5]{fellows2009well} and \cite{2015arxivim}.

The proof of \autoref{main} is now immediate.
\begin{proof}[Proof of \autoref{main}]
  Let $H$ be a graph such that $\excl(H)$ is
  a wqo. Then $H \lctr D_2$, by \autoref{direct}.
On the other hand, if $H \lctr D_2$ then $\excl(H) \subseteq
\excl(D_2)$. Observe that every antichain (respectively  decreasing
sequence) of $(\excl(H), \lctr)$ is an
antichain (respectively a decreasing sequence) of $(\excl(D_2), \lctr)$.
As a consequence of \autoref{wqo} we get that $(\excl(H), \lctr)$ is a
wqo and we are done.
\end{proof}

\section{On canonical antichains}
\label{sec:ca}

In this section, we will use the following result of Ding in order to
prove~\autoref{cactr}. \autoref{fig:dingimg} illustrates the requirements of the lemma.

\begin{figure}[ht]
  \centering
  \begin{tikzpicture}%[every node/.style = black node]
    \begin{scope}
      \foreach \j in {0, ...,2}{
        \pgfmathparse{3.25*\j}
        \begin{scope}[xshift = \pgfmathresult cm]
          \draw (0,0) node (A) {$a_{\j}$};
          \foreach \i in {0, 0.5, ..., 1.5}{
            \draw (-\i, 2) node[draw, color = black, fill = white, rectangle, minimum size = 5pt] (w{\i}) {};
            \draw[-stealth] (A) -- (w{\i});
          }
          \draw (0.75, 2) node {$\dots$};
          \draw[decorate,decoration={brace,amplitude=10pt}] (-1.75, 2.25)
          -- (1, 2.25) node[midway, yshift = 0.75cm, draw = none] {$\mathcal{W}_{\j}$};   
        \end{scope}
      }
      \draw (8.5, 1) node {\Large $\dots$};
    \end{scope}
  \end{tikzpicture}
  \caption{The situation described in \autoref{ding}. Arrows are directed towards larger elements.}
  \label{fig:dingimg}
\end{figure}
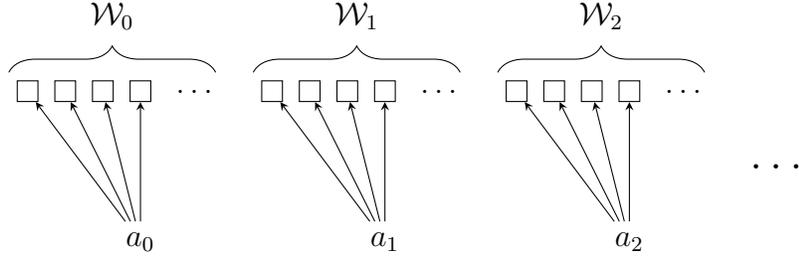

\begin{lemma}[\!\!{\cite[Theorem~1.1]{Ding20091123}}]\label{ding}
  Let $(S, \lleq)$ be a quasi-order, let $\seqb{a_i}_{i\in \N}$ be a
  sequence of elements of $S$ and let $\{\mathcal{W}_i\}_{i\in \N}$ be
  a sequence of sequences of elements of $S$.
If we have
\begin{enumerate}[(i)]
\item $\seqb{a_i}_{i\in \N}$ is a fundamental infinite antichain; and
\item for every $i\in \N$, $\mathcal{W}_i$ is a fundamental infinite antichain;
  and
\item for every $i \in \N$ and every $H \in
  \mathcal{W}_i$, $a_i \lleq H$ and there are no other comparable pairs of elements in $\bigcup_{i \in \N} \mathcal{W}_i\cup \{a_i\}_{i\in \N}$
\end{enumerate}
then $(S, \lleq)$ does not have a canonical antichain.
\end{lemma}

Note that \cite[Theorem~1.1]{Ding20091123} mentions other obstructions to the existence of a canonical antichain, however we will only use that described in \autoref{ding}.
We will now define some sequences of graphs and show that they satisfy
the properties of~\autoref{ding}.

For every $p,q \in \N$, let $W_{p,q}$ be the graph obtained by adding
two non-adjacent dominating vertices to the disjoint union of
$\overline{K}_p$ and $K_{2,q}$ (see \autoref{fig:w34}). These two vertices are called
\emph{poles}, and the two vertices corresponding to the part of
$K_{2,q}$ of size 2 are called \emph{semipoles}. Observe that the
other vertices either have degree two (in which case they are adjacent
to the two poles, only), or have degree four (and they are adjacent to
both poles and both semipoles).

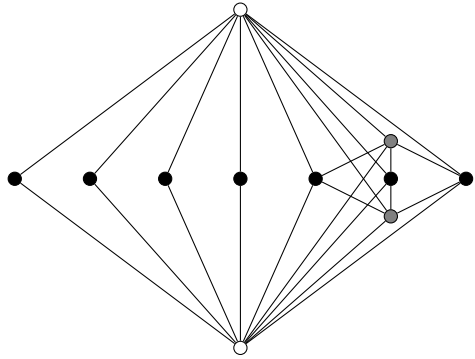
\begin{figure}[ht]
  \centering
  \begin{tikzpicture}[every node/.style = black node]
    % the two poles
    \draw (-3,2.25) node[white node] (p1) {};
    \draw (-3,-2.25) node[white node] (p2) {};
    % the middle layer
    \foreach \i in {0, ..., 6}{
      \draw (-\i, 0) node (n{\i}) {};
      \draw (p1) -- (n{\i}) -- (p2);
    }
    \draw (-1, 0.5) node[fill = black!50] (k1) {}
    (-1, -0.5) node[fill = black!50] (k2) {};
    % drawing the k23
    \draw (n{0}) -- (k1) -- (n{2}) -- (k2) -- (n{0}) (k1) -- (n{1}) -- (k2);
    % connecting it
    \draw (p1) -- (k1) -- (p2);
    \draw (p1) -- (k2) -- (p2);
  \end{tikzpicture}
  \caption{The graph $W_{4,3}$. Poles are drawn in white and semipoles in gray.}
  \label{fig:w34}
\end{figure}

\begin{lemma}\label{comp}
  For every $p,p',q,q'\in \N_{\geq 3}$, there is a model of $W_{p,q}$
  in $W_{p',q'}$ iff $(p,q) = (p',q')$.
\end{lemma}

\begin{proof}
  Let us assume that there is a model $\varphi$ of $W_{p,q}$ in $W_{p',q'}$.
Let $v$ be a vertex of $W_{p,q}$ of degree two. By definition of a
model, $\varphi(v)$
must be a connected subset of degree~2 of $V(W_{p',q'})$. Let us consider the possible choices for this subset.
\smallskip

\noindent \textit{First case:} $\varphi(v)$ is of the form $V(W_{p',q'}) \setminus \{x, y\}$, for some $u,v \in V(W_{p', q'})$. Therefore, $V(W_{p', q'}) \setminus \varphi(v)$ has two vertices.
As $W_{p,q}$ has more than 3 vertices (recall that $p,q \geq 3$) which are mapped by $\varphi$ to disjoint subsets of $V(W_{p',q'})$, this case is not possible.
\smallskip

\noindent \textit{Second case:} $\varphi(v)$ is the subset of vertices inducing the subgraph $K_{2,q'}$ used in the construction of $W_{p', q'}$.
Observe that the poles and semipoles of $W_{p,q}$ (4 vertices in total, as they are distinct from $v$) all have degree at least 3. Moreover, $W_{p', q'} \setminus \varphi(v)$ is isomorphic to $K_{2,p'}$. As every connected subset of degree at least 3 of $K_{2,p'}$ must contain a pole of this graph, there are at most two such subsets that are disjoint. This contradicts the fact that the images by $\varphi$ of the four poles and semipoles of $W_{p,q}$ are disjoint connected subsets of degree at least 3 of $W_{p', q'} \setminus \varphi(v)$. Therefore, this case is not possible neither.
\smallskip

\noindent \textit{Third case:} $\varphi(v) = \{x\}$ for some vertex $x\in V(W_{p',q'}) $ of degree~2.
As we reach this case for every choice of a vertex of degree 2 of $W_{p,q}$, we deduce $p \leq p'$. The same argument applied to vertices
of degree 4 yields~$q \leq q'$.
Let us now consider poles and semipoles.

Let $u$ be a pole. Observe that according to the above remarks,
$\varphi(u)$ must be adjacent to vertices of degree two, so it should
contain a pole of $W_{p,q}$.
If $\varphi(u)$ contains in addition a vertex of degree 2 or 4 of
$W_{p,q}$, then $\varphi(u)$ is dominating. This is not possible since
$u$ is not dominating, therefore $\varphi(u) = \{v\}$ for some pole~$v$ of~$W_{p',q'}$.
Let us now assume that $u$ is a semipole of $W_{p,q}$. As previously, the
above remarks imply that $\varphi(u)$ is adjacent to vertices of
degree 4 of $W_{p',q'}$. Hence $\varphi(u)$ contains a semipole of
$W_{p',q'}$ (it cannot contain a pole as both belong to the image of
poles of $W_{p,q}$). Therefore each semipole of $W_{p,q}$ is sent to a
subset of $V(W_{p',q'})$ containing a semipole. Observe that
$\varphi(u)$ cannot contain a vertex of degree two otherwise it would
not be connected. Besides, it cannot contain a vertex of degree 4
otherwise it would be adjacent to the image by $\varphi$ of the other semipole of
$W_{p,q}$. Consequently $\varphi(u)$ contains a semipole of
$W_{p',q'}$ and no other vertex.
We proved that for every $u \in V(W_{p,q})$, the set $\varphi(u)$ is a
singleton. Therefore $|V(W_{p,q})|  = |V(W_{p',q'})|$.
 Given that $p \leq p'$ and $q \leq q'$ (as proved above), this is possible
only if $p = p'$ and $q = q'$. This concludes the proof.
\end{proof}

\begin{corollary}
  $\{W_{p,q}\}_{p,q \geq 3}$ is an antichain for $\lctr$.
\end{corollary}

For every $i \in \N_{\geq 3}$, let $\mathcal{W}_i = \{W_{i,q}\}_{q \in \N_{\geq 3}}$.

\begin{lemma}\label{kpp1}
  For every $p,q,r \in \N_{\geq 3}$, we have $K_{2,r} \lctr W_{p,q}$ iff $r = p+1$.
\end{lemma}

\begin{proof}
Let us consider a contraction model $\varphi$ of $K_{2,r}$ in $W_{p,q}$.
We call $X$ the vertices of $W_{p,q}$ inducing the $K_{2,q}$ used in the construction of this graph.
Let us consider a vertex $u$ of degree 2 of $K_{2,r}$. Exactly as in the proof of \autoref{comp}, there are three possible choices for $\varphi(u)$. For the same reason as in this proof, the case where $\varphi(u) = V(W_{p,q}) \setminus \{x,y\}$ (for some $x,y \in V(W_{p,q})$) is not possible.
Therefore, either $\varphi(u) = X$, or $\varphi(u) = \{x\}$ for some $x \in V(W_{p,q})$ of degree 2. Since this holds for every vertex of degree 2 of $K_{2,r}$, and as $W_{p,q}$ has exactly $p$ vertices of degree 2, we deduce that $r\leq p+1$.

Let $v,w$ denote the poles of $K_{2,r}$. Because of the observation above and of the definition of a contraction model, each of $\varphi(v)$ and $\varphi(w)$ must be adjacent to some vertex of degree 2 of $W_{p,q}$ and these sets should not be adjacent. The only possible choice for them is to let $\varphi(v)$ be the singleton containing one pole of $W_{p,q}$ and $\varphi(w)$ be the singleton containing the other pole. Observe that, in the case where $p+1>r$, either one vertex of degree 2 of $W_{p,q}$ or a vertex of $X$ does not belong to the image of $\varphi$. This contradicts the definition of a model, hence this case is not possible.

The only remaining case is thus $r = p+1$.
Observe that $X$ induces a connected subgraph.
It is not hard to see that contracting $X$ to a single vertex yields~$K_{2,p+1}$.
\end{proof}

\begin{observation}
  Let $p,q \in \N_{\geq 3}$. There is no induced path on four vertices in~$W_{p,q}$, neither in $K_{2,p}$.
\end{observation}
Then we successively deduce the following consequences.

\begin{corollary}
For every $p,q \in \N_{\geq 3}$, none of the graphs $W_{p,q}$ and $K_{2,p}$ contains
the gem as induced minor.
\end{corollary}

\begin{corollary}\label{wiwqo}
  No graph of $\incl(\mathcal{W}_i)$ and of $\incl(\{K_{2,p}\}_{p \in \N_{\geq 3}})$ contains the gem as induced minor, for
  every $i\in \N_{\geq 3}$.
\end{corollary}

The following observation will allow us to use \autoref{gemfree},
which deals with induced minors.
\begin{observation}\label{imctr}
  Let $H$ and $G$ be two graphs. If both of them have a dominating vertex, then $H$ is a contraction of $G$ iff $H$ is an induced minor of~$G$.
\end{observation}

\begin{lemma}[\!\!\cite{2015arxivim}]\label{gemfree}
Graphs not containing the gem as induced minor are wqo by the induced minor relation.
\end{lemma}

 The following corollaries are direct consequences of \autoref{gemfree},
 \autoref{imctr} and~\autoref{wiwqo}.
 \begin{corollary}\label{k2wqo}
   $\incl(\{K_{2,p}\}_{p \in \N_{\geq 3}})$ is wqo by $\lctr$.
 \end{corollary}

\begin{corollary}\label{domwqo}
  The graphs of $\incl(\mathcal{W}_i)$ with a dominating vertex are wqo
  by~$\lctr$, for every~$i \in \N_{\geq 3}$.
\end{corollary}

\begin{lemma}\label{fund}
  $\mathcal{W}_i$ is a fundamental antichain, for every~$i \in \N_{\geq 3}$.
\end{lemma}

\begin{proof}
  Let $i \in \N_{\geq 3}$. We need to show that
  $(\incl(\mathcal{W}_i), \lctr)$ is a wqo.
Let us call \emph{inner edge} every edge of $W_{p,q}$ that is not
incident with a pole, for every~$p,q \in \N_{\geq 3}$.
Observe that if a graph $H \in \incl(\mathcal{W}_i)$ has been obtained
by contracting at least one edge incident with a pole, then $H$ has a
dominating vertex. According to \autoref{domwqo}, these graphs are wqo
by~$\lctr$, therefore we will here consider graphs of
$\incl(\mathcal{W}_i)$ that have been obtained by only contracting
inner edges. We call $\mathcal{I}$ this class.

We first show that $\mathcal{I}$ is the union of the two following classes:
\begin{itemize}
\item the class $\mathcal{I}_0$ of graphs that can be obtained by
  adding two non-adjacent dominating vertices to $\overline{K}_i +
  D_{q}$ for some $q \in \N_{\geq 0}$; and
\item the class $\mathcal{I}_1$ of graphs that can be obtained by
  adding two non-adjacent dominating vertices to $\overline{K}_i +
  S_q$ for some $q \in \N_{\geq 0}$.
\end{itemize}

Again we use the notion of \emph{poles} to denote the two dominating
vertices added to construct graphs of $\mathcal{I}_0$ and
$\mathcal{I}_1$. A \emph{semipole} is either a dominating vertex of
$D_q$ (when dealing with graphs of $\mathcal{I}_0$), or the dominating
vertex of $S_q$ (when dealing with graphs of $\mathcal{I}_1$).

Contracting an inner edge in $W_{i,q}$ clearly yields a graph of
$\mathcal{I}_0$. Now, observe that any further contraction of an edge
connecting a vertex of degree 4 to a semipole gives a graph of
$\mathcal{I}_0$ again. If, on the other hand, we contract the edge
connecting the two semipoles, then we get a graph of
$\mathcal{I}_1$. On a graph of $\mathcal{I}_1$, contracting an edge of the star
(used in the construction of this graph) still gives a graph of~$\mathcal{I}_1$.
Therefore~$\mathcal{I} = \mathcal{I}_0 \cup \mathcal{I}_1$.

 Let us assume that $\mathcal{I}$ is not wqo by $\lctr$. Therefore it
 has an infinite antichain. As $\mathcal{I} = \mathcal{I}_0 \cup
 \mathcal{I}_1$, one of $\mathcal{I}_0$ and $\mathcal{I}_1$ (at least)
 has an infinite antichain. Let $\mathcal{A}$ be such an infinite
 antichain.

 We now look at vertices of graphs of $\mathcal{A}$ that are neither
 poles, nor semipoles, nor have degree 2. These vertices are the
 vertices of degree 2 of the copy of $D_{q}$ or the vertices of degree
 one of the copy of $S_q$ used in the construction of the graphs of
 $\mathcal{A}$ (depending whether $\mathcal{A} \subseteq
 \mathcal{I}_0$ or $\mathcal{A} \subseteq \mathcal{I}_1$). We call
 them \emph{inner vertices}.

Let $A$ and $A'$ be two graphs of $\mathcal{A}$ such that $A$ has less
inner vertices than $A'$. These graphs exist since the elements of
$\mathcal{A}$ are distinct. Let $q$ be the number
 of inner vertices of $A$ and $q'$ the one of~$A'$.

In both cases $\mathcal{A} \subseteq \mathcal{I}_0$ and $\mathcal{A}
\subseteq \mathcal{I}_1$ we can obtain $A$ from $A'$ by contracting
$q'-q$ inner vertices of $A'$ to a~semipole.
This contradicts the fact that $\mathcal{A}$ is an antichain.
Therefore $(\mathcal{I}, \lctr)$ is a wqo. This implies that
$\mathcal{W}_i$ is fundamental, as required.
\end{proof}

We are now ready to prove \autoref{cactr}.
\begin{proof}[Proof of \autoref{cactr}]
  Let $A_i = K_{2, i+1}$ for every $i \in \N_{\geq 3}$.

By the virtue of \autoref{k2wqo}, $\{A_i\}_{i\in \N_{\geq 3}}$ is a fundamental antichain, as well as $\mathcal{W}_i$, for every $i \in \N_{\geq 3}$, according to \autoref{fund}.
Also, for every $i \in \N_{\geq 3}$, we have $A_i \lctr H$ for every $H \in \mathcal{W}_i$ (\autoref{kpp1}) and there are no other comparable pairs of elements in $\bigcup_{i \in \N_{\geq 3}} \mathcal{W}_i \cup \{A_i\}_{i \in \N_{\geq 3}}$ (\autoref{comp} and \autoref{kpp1}).

Hence these sequences of graphs satisfy the requirements of \autoref{ding}, which implies that there is no canonical antichain for the contraction relation.
\end{proof}

\section*{Acknowledgements}
\addcontentsline{toc}{section}{Acknowledgements}

The authors thank Jarosław Błasiok for
inspiring discussions about the topic of this paper.

%\bibliography{ctr-wqo} \bibliographystyle{plain}

\end{document}